\documentclass[12pt]{amsart}
\usepackage{graphicx}
\newtheorem{thm}{Theorem}[section]

\newtheorem{lem}[thm]{Lemma}

\theoremstyle{definition}
\newtheorem{defn}[thm]{Definition}
\theoremstyle{remark}
\newtheorem{rem}[thm]{Remark}
\numberwithin{equation}{section}

\begin{document}

\title[Semiconservative random walks in weak sense]{Semiconservative random walks in weak sense}%
\author{Vyacheslav M. Abramov}%
\address{Monash University, School of Mathematical Sciences, Wellington Road, Clayton Campus, Clayton, Victoria-3800, Australia}%

\email{vabramov126@gmail.com}%

\begin{abstract}
Conservative and semiconservative random walks in $\mathbb{Z}^d$ were introduced and studied in [V.M. Abramov, J. Theor. Probab. (2017). https://doi.org/10.1007/s10959-017-0747-3]. In the present paper, we extend these concepts for homogeneous random walks in $\mathbb{R}^d$ introducing semiconservative random walks in weak sense and
construct such a family of random walks in $\mathbb{R}^d$.
\end{abstract}
\maketitle

\section{Introduction, definitions, formulation of the problem and the main result}
\subsection{History and motivation of the study}
Let $\mathcal{A}$ be a set of vector-valued parameters, and $d$ is an integer value. Then the triple $\{\mathbf{S}_t,\mathcal{A}, d\}$ is said to specify a family of random walks, where $d$ is the dimension of random walks of the family. A random walk of the family is denoted $\mathbf{S}_t(a)$, where $a\in\mathcal{A}$, and
$$
\mathbf{S}_t(a)=\left(S_t^{(1)}(a), S_t^{(2)}(a),\ldots, S_t^{(d)}(a)\right),
$$
where $S_t^{(i)}(a)$ the $i$th component of the vector $\mathbf{S}_t(a)$. Denote by $$\|\mathbf{S}_t(a)\|=\sum_{i=1}^{d}\big|S_t^{(i)}(a)\big|$$ the $l_1$-norm of $\mathbf{S}_t(a)$.

Some of the families of random walks in $\mathbb{Z}^d$ have been considered in \cite{A}. One of them, the family of symmetric random walks, is defined by the recursion
\begin{eqnarray}
  \mathbf{S}_0(a) &=& \mathbf{0}, \label{eq.0.1+}\\
  \mathbf{S}_t(a) &=& \mathbf{S}_{t-1}(a)+\mathbf{e}_t(a), \quad t=1,2,\ldots,\label{eq.0.2+}
\end{eqnarray}
where $\mathbf{0}$ is $d$-dimensional vector of zeros, and the vector $\mathbf{e}_t(a)$ is one of the $2d$ randomly chosen vectors $\{\pm\mathbf{1}_i, i=1,2,\ldots,d\}$ independently of the history as each other as follows. The probability that the vector $\mathbf{1}_i$ will be chosen, which is the same for the vector $(-\mathbf{1}_i)$, is equal to $\alpha_i>0$, and $2\sum_{i=1}^d\alpha_i=1$. Then, $a=(\alpha_1,\alpha_2,\ldots,\alpha_d)\in\mathcal{A}$ is the $d$-dimensional parameter that characterizes the family of random walks. It was shown in \cite{A} that the family of symmetric random walks is $(\mathcal{A},d)$-conservative, which, according to Definition 1.1 in \cite{A}, means as follows.

Let $t_{n,1}(a)$, $t_{n,2}(a)$,\ldots denote the sequence of times for the consecutive events $\{\|\mathbf{S}_{t_{n,j}(a)}(a)\|=n\}$, $j=1,2,\ldots$.  Then, for any $a_1\in\mathcal{A}$ and $a_2\in\mathcal{A}$ and any $n\geq0$ the corresponding sequences of times are defined by similar way, and
\begin{equation}\label{eq.0.3+}
\begin{aligned}
&\lim_{j\to\infty}\mathsf{P}\{\|\mathbf{S}_{t_{n,j}(a_1)+1}(a_1)\|=n+1~|~\|\mathbf{S}_{t_{n,j}(a_1)}(a_1)\|=n\}\\
&=\lim_{j\to\infty}\mathsf{P}\{\|\mathbf{S}_{t_{n,j}(a_2)+1}(a_2)\|=n+1~|~\|\mathbf{S}_{t_{n,j}(a_2)}(a_2)\|=n\}.
\end{aligned}
\end{equation}
Note that the time instant $t_{n,j}(a)$, $a\in\mathcal{A}$, for any given $j$ exists with some positive probability $p_{n,j}(a)$, and the following two cases are possible: either $p_{n,j}(a)\equiv1$ for all $j$ when the random walk is recurrent, i.e. $d\leq2$, or $p_{n,j}(a)$ vanishes as $j$ increases to infinity when the random walk is transient, i.e. $d>2$. Thus, the sequences of times $t_{n,1}(a_1)$, $t_{n,2}(a_1)$,\ldots and $t_{n,1}(a_2)$, $t_{n,2}(a_2)$,\ldots, for which the limits in \eqref{eq.0.3+} are defined, are assumed to be given either with probability 1 or with probability 0. So, the definition implies that the original probability space is generated by the variety of events that include all those having probability zero in limits. Then, \eqref{eq.0.3+} is properly defined.

Along with symmetric random walks defined by \eqref{eq.0.1+} and \eqref{eq.0.2+}, there are some other families of random walks considered in \cite{A}. Being an extension of the family of symmetric random walks, those families of random walks are characterized as semiconservative random walks. Following \cite{A}, a family of random walks $\{\mathbf{S}_t,\mathcal{A}, d\}$ is called $(\mathcal{A},d)$-semiconservative, if there exists $a^*\in\mathcal{A}$ such that for any $a\in\mathcal{A}$ either
\begin{equation}\label{eq.0.4+}
\begin{aligned}
&\lim_{j\to\infty}\mathsf{P}\{\|\mathbf{S}_{t_{n,j}(a)+1}(a)\|=n+1~|~\|\mathbf{S}_{t_{n,j}(a)}(a)\|=n\}\\
&\leq\lim_{j\to\infty}\mathsf{P}\{\|\mathbf{S}_{t_{n,j}(a^*)+1}(a^*)\|=n+1~|~\|\mathbf{S}_{t_{n,j}(a^*)}(a^*)\|=n\},
\end{aligned}
\end{equation}
or
\begin{equation}\label{eq.0.5+}
\begin{aligned}
&\lim_{j\to\infty}\mathsf{P}\{\|\mathbf{S}_{t_{n,j}(a)+1}(a)\|=n+1~|~\|\mathbf{S}_{t_{n,j}(a)}(a)\|=n\}\\
&\geq\lim_{j\to\infty}\mathsf{P}\{\|\mathbf{S}_{t_{n,j}(a^*)+1}(a^*)\|=n+1~|~\|\mathbf{S}_{t_{n,j}(a^*)}(a^*)\|=n\}.
\end{aligned}
\end{equation}
The sequences of times $t_{n,j}(a)$ and $t_{n,j}(a^*)$ ($j=1,2,\ldots$) used in \eqref{eq.0.4+} and \eqref{eq.0.5+} are defined similarly to those in \eqref{eq.0.3+}.

The concept of conservative random walks provides a clear intuition for recurrence and transience of the families of random walks. For instance, symmetric random walks have the same classification as simple (P\'olya) random walks, since all of them belong to the same class of conservative random walks. The role of semiconservative random walks is important, since this class of random walks is wide enough to provide the classification for different families of state-dependent random walks (see \cite{A}).

Unfortunately, the classes of conservative and semiconservative random walks are relatively narrow and cannot be specified for many families of random walks. Therefore, in \cite{AA} another classification of random walks has been suggested. In that classification, the families of random walks are described by the couple $(\mathbf{S}_t,\mathcal{A})$ (without indication of the dimension $d$) and for all $a\in\mathcal{A}$ are characterized by the limit relation
\begin{equation}\label{eq.0.6+}
\lim_{n\to\infty}\left(\frac{\mathsf{P}\{\|\mathbf{S}_{t+1}(a)\|=n+1~|~\|\mathbf{S}_{t}(a)\|=n\}}{\mathsf{P}\{\|\mathbf{S}_{t+1}(a)\|=n-1~|~\|\mathbf{S}_{t}(a)\|=n\}}\right)^n
=\mathrm{e}^\psi,
\end{equation}
which is assumed to exist.
The parameter $\psi$ on the right-hand side of \eqref{eq.0.6+} is called \textit{index} of the family of random walks. For instance, the index of the family of conservative random walks in $\mathbb{Z}^d$ is equal to $d-1$, i.e. $\psi+1=d$, and this relation between $\psi$ and $d$ is satisfied for many families of random walks (see Lemma 4.2 and its application in Section 4 in \cite{A}). So, for $\psi\leq1$ the family of random walks is recurrent and for $\psi>1$ transient. In some cases, however, the index $\psi$ does not satisfy the aforementioned relation and serves as a fractional characteristic of the random walk. For classes of $\psi$-random walks see \cite{AA}. Here we provide a simple example of the known family of one-dimensional random walks \cite{Kl}. Consider a family of one-dimensional random walks $(S_t,\mathcal{A})$, where an element $a\in\mathcal{A}$ is an infinite-dimensional vector denoted by ($\alpha_1$, $\alpha_2$,\ldots), and satisfying the property $\lim_{j\to\infty}\alpha_j=\alpha^*$. Let $S_0(a)=0$, $S_1(a)=\pm1$ each with probability half, $S_t(a)=S_{t-1}(a)+e_t(a)$, $t\geq2$, where $e_t(a)$ takes the values $\pm1$, and the distribution of $|S_t(a)|$ is defined by the following conditions. If $|S_t(a)|>0$, then
\begin{eqnarray*}
\mathsf{P}\{|S_{t+1}(a)|=|S_{t}(a)|+1\}&=&\frac{1}{2}+\frac{\alpha_{|S_t(a)|}}{|S_t(a)|},\\
\mathsf{P}\{|S_{t+1}(a)|=|S_{t}(a)|-1\}&=&\frac{1}{2}-\frac{\alpha_{|S_t(a)|}}{|S_t(a)|}.
\end{eqnarray*}
Otherwise, if $S_t(a)=0$, then $\mathsf{P}\{|S_{t+1}(a)|=1\}=1$. The values $\alpha_n$ are assumed to satisfy the condition $|\alpha_n|<\min\{C, (1/2)n\}$ for some $C>0$.
Apparently, in this case $\psi=4\alpha^*$, and the family of random walks is recurrent if $\alpha^*\leq1/4$. Otherwise, it is transient. In \cite{Kl}, the result for this example in slightly different formulation was established in the framework of the theory of stochastic difference equations.

The characterization of random walks with index $\psi$ seems can be helpful for many existing random walks problems (see Chapter 2 of \cite{L-P}). However, this type of characterization is hard to extend for real-valued random walks. In particular, it is hard to do even for one-dimensional random walks. Therefore, doing step down, we develop the conservative and semiconservative random walks concept and provide new knowledge for some families of real-valued random walks.

\subsection{Formulation of the problem, definitions} We consider the family of random walks $\{\mathbf{S}_t, \mathcal{A},d\}$, where a $d$-dimensional random walk of the family is denoted $\mathbf{S}_t(a)$ for $a\in\mathcal{A}$ and defined by the recursion
\begin{eqnarray}
  \mathbf{S}_0(a) &=& \mathbf{0}, \label{eq.1.1}\\
  \mathbf{S}_t(a) &=& \mathbf{S}_{t-1}(a)+\mathbf{x}_t(a), \quad t=1,2,\ldots,\label{eq.1.2}
\end{eqnarray}
where the vectors $\mathbf{S}_t(a)$ and $\mathbf{x}_t(a)$ are represented as
$$
\mathbf{S}_t(a)=\left(S_t^{(1)}(a), S_t^{(2)}(a),\ldots, S_t^{(d)}(a)\right)
$$
and, respectively,
\begin{equation*}\label{eq.1.3}
\mathbf{x}_t(a)=\left(x_t^{(1)}(a), x_t^{(2)}(a),\ldots, x_t^{(d)}(a)\right).
\end{equation*}
The sequence of the vectors $\mathbf{x}_1(a)$, $\mathbf{x}_2(a)$, \ldots is assumed to be mutually independent for any $a\in\mathcal{A}$, and the coordinates
$$x_t^{(1)}(a), x_t^{(2)}(a),\ldots, x_t^{(d)}(a)$$
of the vector $\mathbf{x}_t(a)$ all are real random variables taking at least two distinct values and assumed to be independent for any $t$ and $a$.
The definition below extends the concepts of conservative and semiconservative random walks for families of random walks in $\mathbb{R}^d$ given by \eqref{eq.1.1} and \eqref{eq.1.2}.

\begin{defn}\label{D1} A family of random walks $\{\mathbf{S}_t,\mathcal{A}, d\}$ is called $(\mathcal{A}, d)$-conservative, if for any $a_1\in\mathcal{A}$ and $a_2\in\mathcal{A}$ and all $z>0$
\begin{equation}\label{eq.0.1}
\begin{aligned}
&\lim_{t\to\infty}\mathsf{P}\{\|\mathbf{S}_t(a_1)\|>z~|~\|\mathbf{S}_{t-1}(a_1)\|=z\}\\ &=\lim_{t\to\infty}\mathsf{P}\{\|\mathbf{S}_t(a_2)\|>z~|~\|\mathbf{S}_{t-1}(a_2)\|=z\},
\end{aligned}
\end{equation}
and $(\mathcal{A}, d)$-semiconservative, if there exists $a^*\in\mathcal{A}$ such that for all $a\in\mathcal{A}$ and $z>0$ either
\begin{equation}\label{eq.0.2}
\begin{aligned}
&\lim_{t\to\infty}\mathsf{P}\{\|\mathbf{S}_t(a^*)\|>z~|~\|\mathbf{S}_{t-1}(a^*)\|=z\}\\ &\leq\lim_{t\to\infty}\mathsf{P}\{\|\mathbf{S}_t(a)\|>z~|~\|\mathbf{S}_{t-1}(a)\|=z\},
\end{aligned}
\end{equation}
or
\begin{equation}\label{eq.0.3}
\begin{aligned}
&\lim_{t\to\infty}\mathsf{P}\{\|\mathbf{S}_t(a^*)\|>z~|~\|\mathbf{S}_{t-1}(a^*)\|=z\}\\ &\geq\lim_{t\to\infty}\mathsf{P}\{\|\mathbf{S}_t(a)\|>z~|~\|\mathbf{S}_{t-1}(a)\|=z\}.
\end{aligned}
\end{equation}
Here in \eqref{eq.0.1}, \eqref{eq.0.2} and \eqref{eq.0.3}, $\|\mathbf{S}_t(a)\|=\sum_{i=1}^{d}\big|S_t^{(i)}(a)\big|$ denotes the $l_1$-norm of the vector $\mathbf{S}_t(a)$.
\end{defn}

Note that unlike the case of integer-valued random walks, all conditional probabilities that appear here in \eqref{eq.0.1}, \eqref{eq.0.2} and \eqref{eq.0.3} and later in \eqref{eq.0.4} and \eqref{eq.0.5} are given with respect to the events having probability zero.

\smallskip

In the definition below, we introduce semiconservative random walks in weak sense.

\begin{defn}\label{D2}
A family of random walks $\{\mathbf{S}_t,\mathcal{A}, d\}$ is called $(\mathcal{A}, d)$-semiconservative \textit{in weak sense} if there exists $a^*\in\mathcal{A}$ such that for all $a\in\mathcal{A}$ either
\begin{equation}\label{eq.0.4}
\begin{aligned}
&\lim_{t\to\infty}\mathsf{P}\{\|\mathbf{S}_t(a^*)\|>z~|~\|\mathbf{S}_{t-1}(a^*)\|=z\}\\
&\leq\lim_{t\to\infty}\mathsf{P}\{\|\mathbf{S}_t(a)\|>z~|~\|\mathbf{S}_{t-1}(a)\|=z\}
\end{aligned}
\end{equation}
or
\begin{equation}\label{eq.0.5}
\begin{aligned}
&\lim_{t\to\infty}\mathsf{P}\{\|\mathbf{S}_t(a^*)\|>z~|~\|\mathbf{S}_{t-1}(a^*)\|=z\}\\
&\geq\lim_{t\to\infty}\mathsf{P}\{\|\mathbf{S}_t(a)\|>z~|~\|\mathbf{S}_{t-1}(a)\|=z\}
\end{aligned}
\end{equation}
is satisfied for all $z\geq z^*(a)$, where $z^*(a)$ is some nonnegative value depending on $a$.
\end{defn}

In the present paper, we construct a family of semiconservative random walks in weak sense in $\mathbb{R}^d$. We first construct such a family in $\mathbb{Z}^d$. For this purpose, we reduce the problem to the specified Markovian queueing system with multiple classes, positive and negative arrivals and reflections, and derive the system of equations for the system states. Application of queueing theory is provided in the following way. We assume that a random walk spends in any of its states exponentially distributed time, thus reducing it to continuous time Markov process. The advantage of application of this approach is that one can quite easily derive the system of Chapman-Kolmogorov differential equations, the theory and results of which are well-established. As well, the object of the study is the conditional distribution of the norm of the random walk. Since the norm of an original and reflected random walk have the same distribution, we choose to study the reflected random walk, the increments of which can be naturally interpreted in terms of queueing theory. Then, the results obtained for families of integer-valued random walks are first extended for series of rational-valued random walks, and then, in the limiting case, are established for families of real-valued random walks. In our proof, we use properties of random variables having same means and different variances to achieve the required inequalities of Definition \ref{D2}.

\subsection{Notation and formulation of the main result}
Let $$F^{(i)}(x; a)=\mathsf{P}\big\{x_t^{(i)}(a)\leq x\big\}, \quad i=1,2,\ldots,d, \quad a\in\mathcal{A}.$$

\smallskip
The main result of the paper is the following theorem.

\begin{thm}\label{T1}
Let $\mathcal{A}$ be a set of all vectors $a$ $=\big(a^{(1)}$, $a^{(2)}$,\ldots, $a^{(d)}\big)$ with positive components satisfying the condition
$$
\sum_{i=1}^{d}a^{(i)}=d.
$$
Then, the family of random walks $\{\mathbf{S}_t, \mathcal{A},d\}$ is $(\mathcal{A},d)$-semiconservative in weak sense,
if the following conditions are satisfied:
$$
F^{(i)}(x; a)=F\left(\frac{x}{a^{(i)}}\right), \quad i=1,2,\ldots,d,\leqno(i)
$$

\noindent
(ii) \quad $F(x)$
is a probability distribution function of zero mean and finite variance  random variable.
\end{thm}

\begin{rem}\label{R1}
Note that in the case when $\mathbf{1}=(\underbrace{1,1,\ldots,1}_{d \ \text{units}})\in\mathcal{A}$, the random walk $\mathbf{S}_t(\mathbf{1})$ belongs to the family $\{\mathbf{S}_t, \mathcal{A},d\}$ of the random walks considered in Theorem \ref{T1}. The aforementioned vector $\mathbf{1}$ is associated with the vector $a^*$ given in Definition \ref{D2}, and in this case $F^{(i)}(x; \mathbf{1})=F(x)$, $ i=1,2,\ldots,d$.
We will call this case \textit{simple}. Any other case with $a\in\mathcal{A}$ will be called \textit{regular} case.
\end{rem}

\subsection{Outline of the paper} The rest of the paper is structured as follows. In Section \ref{S2}, we model a reflected random walk in $\mathbb{Z}^d$ as a specifically defined multiclass queueing system with positive and negative arrivals and reflections. In Section \ref{S3}, we derive basic equations for a queueing system. In Section \ref{S4} containing five subsections, we prove the main theorem of the paper.  In Section \ref{S4.1}, we describe the plan of the proof. In Sections \ref{S4.2} and \ref{S4.3} we represent the formulae for the required limiting conditional probabilities in simple and regular cases.
Section \ref{S4.4} contains the proof of the statement of the theorem in the case when the family of random walks is defined in $\mathbb{Z}^d$. In Section \ref{S4.5}, we extend the proof for the family of random walks defined in $\mathbb{R}^d$. In short Section \ref{S5}, we conclude the paper.

\section{The queueing model for a random walk}\label{S2}
In this section we model a random walk. Let $$\breve{\mathbf{S}}_t(a)=\left(\breve{S}_t^{(1)}(a), \breve{S}_t^{(2)}(a),\ldots,\breve{S}_t^{(d)}(a)\right)$$ denote the reflected random walk with respect to the original random walk $\mathbf{S}_t(a)$. With the initial value $\breve{\mathbf{S}}_0(a)=\mathbf{0}$, it is defined as follows. For $i=1,2,\ldots,d$,
$$
\breve{S}_t^{(i)}(a)=\left|\breve{S}_{t-1}^{(i)}(a)+x_t^{(i)}(a)\right|, \quad t\geq1.
$$
Since $\|\breve{\mathbf{S}}_t(a)\|=\|\mathbf{S}_t(a)\|$, the problem reduces to study the reflected random walk $\breve{\mathbf{S}}_t(a)$. Similarly to \cite{A}, assume that a random walk spends exponentially distributed time in any of its states prior moving to another state. Then, a reflected random walk can be modelled as the Markovian single-server queueing system presented below. In the description of the system a series parameter $a$ is omitted, since the queueing model describes a unique (common) random walk. However, in the places where it is clearly necessary, the parameter $a$ will be added.

Consider the following queueing system with $d$ classes. (We will not use the term \textit{customer class}, since the most important characteristic that is studied here is the \textit{workload} of the class $i$ arrivals, while the number of class $i$ customers is particular with respect to the workload, and it may be not defined in general.   Hence, we prefer to use the word \textit{arrival} rather than \textit{customer} and \textit{class} or \textit{arrival class} rather than \textit{customer class}.) Assume that all $d$ classes arrive simultaneously in the system according to Poisson input with rate 1.
An arrival class can be \textit{positive} or \textit{negative}. A more detailed explanation about it is provided below.

An arrival includes a batch of all $d$ classes, and the $j$th arrival of class $i$, if it is positive, is characterized by a random quantity $B_j^{(i)}$, and if it is negative, then its random quantity is denoted by $\widetilde{B}_j^{(i)}$. The random variables $B_j^{(i)}$ and $\widetilde{B}_j^{(i)}$ are real nonnegative random variables in general.
That is, the total quantity of the $j$th arrival, if all of them are positive, is $B_j=B_j^{(1)}+B_j^{(2)}+\ldots+B_j^{(d)}$. If, say, all $j$th arrival classes are positive except class 2, then $B_j=B_j^{(1)}+\widetilde{B}_j^{(2)}+\ldots+B_j^{(d)}$.
Each of the $2d$ sequences $\big\{B_1^{(i)}$, $B_2^{(i)},\ldots\big\}$, $\big\{\widetilde{B}_1^{(i)}$,  $\widetilde{B}_2^{(i)}\ldots\}$ ($i=1,2,\ldots,d$) consists of independent identically distributed random variables, and the sequences themselves are mutually independent.
(In the sequel, in the places where the indication of the $j$th arrival is not important, the index $j$ will be omitted.) It is assumed that $\mathsf{E}\widetilde B^{(i)}= \mathsf{E}B^{(i)}$, $i=1,2,\ldots,d$, and positive and negative arrivals of each class  in the arrival process  appear equally likely.

The service times are not considered in this system. Their role is given to negative arrivals, which normally reduce the workload of the system as explained below.
Let $W_j^{(i)}$ denote the total workload of class $i$ before the $j$th arrival. If the $j$th arrival of class $i$ is positive, then at the moment of arrival its total workload becomes equal to $W_j^{(i)}+B_j^{(i)}$. However, if the $j$th arrival of class $i$ is negative, then the workload mechanism is as follows. If $W_j^{(i)}\geq \widetilde B_j^{(i)}$, then, after the service completion, the remaining class $i$ workload is $W_j^{(i)}-\widetilde B_j^{(i)}$. This case is referred to as \textit{ordinary case}. However, if $W_j^{(i)}< \widetilde B_j^{(i)}$, then the following \textit{reflection mechanism} is assumed. The workload preceding the $j$th arrival of class $i$, $W_j^{(i)}$, is discharged, and the value of the workload after the $j$th arrival is changed to the value $\widetilde B_j^{(i)}-W_j^{(i)}$. This case is referred to as \textit{reflection case}.

For a clearer connection between the queueing system with reflections that is described above and a reflected random walk $\breve{\mathbf{S}}_t(a)$, equate
\begin{eqnarray*}
  \mathsf{P}\left\{B^{(i)}(a)\leq x\right\} &=& \mathsf{P}\left\{x_t^{(i)}(a)\leq x~|~x_t^{(i)}(a)\geq0\right\},\\
  \mathsf{P}\left\{\widetilde B^{(i)}(a)\leq x\right\} &=& \mathsf{P}\left\{-x_t^{(i)}(a)\leq x~|~x_t^{(i)}(a)<0\right\},\\
  i=1,2,\ldots, d && (a\in\mathcal{A}).
\end{eqnarray*}

The important particular case of the system is the case when the random variables $B^{(i)}$ and $\widetilde B^{(i)}$, $i=1,2,\ldots, d$, all are integer random variables. In this particular case we assume that  each of the random variables $B^{(i)}$, $i=1,2,\ldots,d$, can take zero value with positive probability, while $\widetilde B^{(i)}$, $i=1,2,\ldots,d$, are assumed to be positive random variables.
For integer-valued random variables $B^{(i)}$ and $\widetilde B^{(i)}$  we use the notation
\begin{eqnarray*}
  r_n^{(i)}(a)=\mathsf{P}\left\{B^{(i)}(a)=n\right\} &=& \mathsf{P}\left\{x_t^{(i)}(a)= n~|~x_t^{(i)}(a)\geq0\right\},\\
  \widetilde{r}_n^{(i)}(a)=\mathsf{P}\left\{\widetilde B^{(i)}(a)=n\right\} &=& \mathsf{P}\left\{x_t^{(i)}(a)=-n~|~x_t^{(i)}(a)<0\right\},\\
  i=1,2,\ldots, d && (a\in\mathcal{A}).
\end{eqnarray*}

\section{Basic equations of queueing process}\label{S3}
In this section, we present the basic equations for the queue-length process in the case where the random variables $B^{(i)}$ and $\widetilde B^{(i)}$, $i=1,2,\ldots, d$, are integer-valued. (As before, we use the parameter $a\in\mathcal{A}$ in the notation for $B^{(i)}$, $\widetilde B^{(i)}$ and other quantities in only places where it is logically required. Otherwise, for compactness of the formula presentations, it is omitted.)
In this case $B_j^{(i)}$ and $\widetilde B_j^{(i)}$ are called batch size of class $i$ customers in the $j$th positive arrival, and, respectively, batch size of class $i$ customers in the $j$th negative arrival.

Let $Q_i(t)$ denote the number of class $i$ customers at time $t$, and with $Q_i(0)=0$  let $$p^{(i)}_n(t)=\mathsf{P}\{Q_i(t)=n\}, \quad n=0,1,\ldots.$$

We have the following Chapman-Kolmogorov system of the equations
\begin{equation}\label{eq.3.1}
\begin{aligned}
&\frac{\mathrm{d}p^{(i)}_n(t)}{\mathrm{d}t}+p^{(i)}_n(t)\\
&=\sum_{l=0}^{n}p^{(i)}_l(t)r^{(i)}_{n-l}+\sum_{l=n+1}^{\infty}p^{(i)}_l(t)\widetilde r^{(i)}_{l-n}+\sum_{l=0}^{\infty}p^{(i)}_l(t)\widetilde r^{(i)}_{l+n}, \quad n\geq1,
\end{aligned}
\end{equation}
\begin{equation}\label{eq.3.2}
\begin{aligned}
&\frac{\mathrm{d}p^{(i)}_0(t)}{\mathrm{d}t}+p^{(i)}_0(t)=p^{(i)}_0(t)r^{(i)}_0+\sum_{l=1}^{\infty}p^{(i)}_l(t)\widetilde r^{(i)}_{l}.
\end{aligned}
\end{equation}

Since $\mathsf{E}\widetilde B^{(i)}=\mathsf{E} B^{(i)}$ for all $i=1,2,\ldots$, then the final probabilities $p^{(i)}_n(\infty)$, $n=0,1,\ldots$, do not exist. However, the normalized quantities
\begin{equation}\label{eq.3.5}
q^{(i)}_n=\lim_{t\to\infty}\frac{p^{(i)}_n(t)}{p^{(i)}_0(t)}, \quad n=0,1,\ldots \quad (i=1,2,\ldots,d)
\end{equation}
do.
In the sequel, we do not use explicit representation \eqref{eq.3.1}, \eqref{eq.3.2} or \eqref{eq.3.5} in direct way.

Our further aim is to study the quantities $q^{(i)}_n(a)$ under the special assumption that
$$F^{(i)}(x, a)=F\left(\frac{x}{a^{(i)}}\right), \quad i=1,2,\ldots, d, \quad a\in\mathcal{A}.$$
Recalling that all $B^{(i)}(a)$ and $\widetilde B^{(i)}(a)$ are assumed to be integer random variables ($i=1, 2,\ldots, d$, $a\in\mathcal{A}$), under the assumption of the theorem we have the following two important properties.

\textit{Property 1}. For $i=1,2,\ldots, d$ and $j=1,2,\ldots,d$ we have
\begin{eqnarray}
r^{(i)}_{a^{(i)}n}(a) &=& r^{(j)}_{a^{(j)}n}(a),\label{eq.3.11}\\
\widetilde{r}^{(i)}_{a^{(i)}n}(a) &=& \widetilde{r}^{(j)}_{a^{(j)}n}(a),\label{eq.3.14}
\end{eqnarray}
Here in \eqref{eq.3.11} and \eqref{eq.3.14} $a=\big(a^{(1)}$, $a^{(2)}$,\ldots, $a^{(d)}\big)\in\mathcal{A}$. The coordinates $a^{(i)}$, $i=1,2,\ldots,d$, of the vector $a$ are rational numbers in general, whereas $a^{(i)}n$, $i=1,2,\ldots,d$, are integer numbers. (A more detailed explanation is given below.)

Let $\mathcal{B}^{(i)}(a)$ denote the set of positive integer values of random variables $B^{(i)}(a)$ and $\widetilde B^{(i)}(a)$. For instance, if the random variable $B^{(i)}(a)$ takes the values $\{0, 2, 4, 6\}$, and the random variable $\widetilde B^{(i)}(a)$ takes the values $\{1, 2, 3, 6\}$, then $\mathcal{B}^{(i)}(a)=\{1, 2, 3, 4, 6\}$. Denote by $\mathsf{gcd}\big(\mathcal{B}^{(i)}(a)\big)$ the greatest common divisor of the values of the set $\mathcal{B}^{(i)}(a)$. Relations \eqref{eq.3.11} and \eqref{eq.3.14} imply
\begin{equation}\label{eq.3.12}
a^{(j)}\mathsf{gcd}\left(\mathcal{B}^{(i)}(a)\right)=a^{(i)}\mathsf{gcd}\left(\mathcal{B}^{(j)}(a)\right),
\end{equation}
where the left- and right-hand sides of \eqref{eq.3.12} are integers. Hence, the index $n$ in \eqref{eq.3.11} and \eqref{eq.3.14} takes the values
\begin{equation}\label{eq.3.13}
n=\frac{l}{a^{(i)}}\mathsf{gcd}\left(\mathcal{B}^{(i)}(a)\right)=\frac{l}{a^{(j)}}\mathsf{gcd}\left(\mathcal{B}^{(j)}(a)\right), \quad l=0,1,\ldots.
\end{equation}

So, from \eqref{eq.3.11},  \eqref{eq.3.14}  and from the basic equations given by \eqref{eq.3.1}, \eqref{eq.3.2} we arrive at
\begin{equation}\label{eq.3.15}
q^{(i)}_{a^{(i)}n}(a) = q^{(j)}_{a^{(j)}n}(a),
\end{equation}
in which the index $n$ is defined by \eqref{eq.3.13} (similarly to that it is defined in \eqref{eq.3.11} and \eqref{eq.3.14}).
Finally, from \eqref{eq.3.13} and \eqref{eq.3.15} we arrive at
\begin{equation}\label{eq.3.19}
q^{(i)}_{l\mathsf{gcd}\left(\mathcal{B}^{(i)}(a)\right)}=q^{(j)}_{l\mathsf{gcd}\left(\mathcal{B}^{(j)}(a)\right)}, \quad l=0,1,\ldots.
\end{equation}

\textit{Property 2}. Let $a_1$ $=\big(a^{(1)}_1,$ $a^{(2)}_1,$ \ldots, $a^{(d)}_1\big)$ and $a_2$ $=\big(a^{(1)}_2,$ $a^{(2)}_2,$ \ldots, $a^{(d)}_2\big)$, $a_1\in\mathcal{A}$, $a_2\in\mathcal{A}$.
It follows from the assumption of the theorem
\begin{eqnarray}
r^{(i)}_{a^{(i)}_1n}(a_1)&=& r^{(i)}_{a^{(i)}_2n}(a_2),\label{eq.3.3}\\
\widetilde r^{(i)}_{a^{(i)}_1n}(a_1) &=& \widetilde r^{(i)}_{a^{(i)}_2n}(a_2),\label{eq.3.4}
\end{eqnarray}
where  $a^{(i)}_1$ and $a^{(i)}_2$ ($i=1,2,\ldots,d$) are generally the rational numbers depending on $a_1$ and $a_2$, whereas $a^{(i)}_1n$ and $a^{(i)}_2n$ are integer numbers.

It follows from \eqref{eq.3.3} and \eqref{eq.3.4}   that
$$a^{(i)}_2\mathsf{gcd}\big(\mathcal{B}^{(i)}(a_1)\big)=a^{(i)}_1\mathsf{gcd}\big(\mathcal{B}^{(i)}(a_2)\big),$$
and since $B^{(i)}(a_1)$ and $B^{(i)}(a_2)$ are integer, then $a^{(i)}_2\mathsf{gcd}\big(\mathcal{B}^{(i)}(a_1)\big)$ and $a^{(i)}_1\mathsf{gcd}\big(\mathcal{B}^{(i)}(a_2)\big)$ are assumed to be integer. Thus, the probability distributions $r^{(i)}_{a^{(i)}_1n}(a_1)$ and $\widetilde r^{(i)}_{a^{(i)}_1n}(a_1)$ on the left-hand sides of \eqref{eq.3.3} and \eqref{eq.3.4}, respectively,  are defined for the integer indices $a^{(i)}_1n$, and the probability distributions $r^{(i)}_{a^{(i)}_2n}(a_2)$ and $\widetilde r^{(i)}_{a^{(i)}_2n}(a_2)$ on the right-hand sides of \eqref{eq.3.3} and \eqref{eq.3.4}, respectively, are defined for the integer indices $a^{(i)}_2n$.
Then, similarly to \eqref{eq.3.13}, the quantities $q^{(i)}_{a^{(i)}_1n}(a_1)$ and $q^{(i)}_{a^{(i)}_2n}(a_2)$ are reckoned to be defined for the indices
\begin{equation}\label{eq.3.7}
n=\frac{l}{a^{(i)}_1}\mathsf{gcd}\left(\mathcal{B}^{(i)}(a_1)\right)=\frac{l}{a^{(i)}_2}\mathsf{gcd}\left(\mathcal{B}^{(i)}(a_2)\right), \quad l=0,1,\ldots.
\end{equation}
So, from \eqref{eq.3.3} and \eqref{eq.3.4} and the basic equations given by \eqref{eq.3.1}, \eqref{eq.3.2} we arrive at
\begin{equation}\label{eq.3.6}
q^{(i)}_{a^{(i)}_1n}(a_1) = q^{(i)}_{a^{(i)}_2n}(a_2),
\end{equation}
in which the index $n$ is defined by \eqref{eq.3.7} (similarly to that it is defined in \eqref{eq.3.3} and \eqref{eq.3.4}). Finally, from \eqref{eq.3.7} and \eqref{eq.3.6} we arrive at
\begin{equation}\label{eq.3.8}
q^{(i)}_{l\mathsf{gcd}\left(\mathcal{B}^{(i)}(a_1)\right)}(a_1) = q^{(i)}_{l\mathsf{gcd}\left(\mathcal{B}^{(i)}(a_2)\right)}(a_2), \quad l=0,1,\ldots.
\end{equation}

\begin{rem}\label{R2}
Note that $r^{(i)}_k(\mathbf{1})=r^{(j)}_k(\mathbf{1})$ and $\widetilde r^{(i)}_k(\mathbf{1})=\widetilde r^{(j)}_k(\mathbf{1})$ for all $i=1,2,\ldots, d$ and $j=1,2,\ldots, d$.
\end{rem}

\section{Proof of Theorem \ref{T1}}\label{S4}

In this section we prove the main theorem.

\subsection{Plan of the proof}\label{S4.1} We first prove the theorem under the assumption that $\mathbf{x}_t(a)$, $a\in\mathcal{A}$, is a vector with integer components, that is, $B^{(i)}(a)$, $i=1,2,\ldots,d$, are assumed to be integer random variables. Then, being proved under this assumption, the results can be extended to the case when the components of the vector are rational numbers, and then to the limiting case when the aforementioned random variables are assumed to be continuous. In the case when $B^{(i)}(a)$, $i=1,2,\ldots,d$, are integer random variables, the proof is provided by the same scheme as the related part of the proof of Theorem 2.1 in \cite{A}. First, we derive an expression for the limiting probability
\begin{equation}\label{eq.4.1}
\lim_{t\to\infty}\mathsf{P}\{\|\mathbf{S}_t(\mathbf{1})\|>z~|~\|\mathbf{S}_{t-1}(\mathbf{1})\|=z\}
\end{equation}
Then we compare that expression with
\begin{equation}\label{eq.4.2}
\lim_{t\to\infty}\mathsf{P}\{\|\mathbf{S}_t(a)\|>z~|~\|\mathbf{S}_{t-1}(a)\|=z\}, \quad a\in\mathcal{A}.
\end{equation}

\subsection{Presentation of \eqref{eq.4.1} (simple case)}\label{S4.2} Recall that in the simple case we derive an expression for \eqref{eq.4.1}. The derivation is based on combinatorial arguments.
 The expression for the limiting probability in \eqref{eq.4.1} is presented in the form
 \begin{equation}\label{eq.4.3}
\lim_{t\to\infty}\mathsf{P}\{\|\mathbf{S}_t(\mathbf{1})\|>z~|~\|\mathbf{S}_{t-1}(\mathbf{1})\|=z\}=\frac{N(\mathbf{1})}{D(\mathbf{1})}.
\end{equation}
For the denominator $D(\mathbf{1})$ we have the expression
\begin{equation}\label{eq.4.4}
D(\mathbf{1})=\sum_{\substack {k_1+k_2+\ldots+k_d=z/\mathsf{gcd}(\mathcal{B}^{(i)}(\mathbf{1}))\\k_i\geq0, \quad i=1,2,\ldots,d}}\prod_{i=1}^{d}q^{(i)}_{k_i\mathsf{gcd}(\mathcal{B}^{(i)}(\mathbf{1}))}.
\end{equation}
Note that the index multiplier $\mathsf{gcd}(\mathcal{B}^{(i)}(\mathbf{1}))$ that appears on the right-hand side of \eqref{eq.4.4} is a same value for all $i=1,2,\ldots,d$, that is, $\mathsf{gcd}(\mathcal{B}^{(i)}(\mathbf{1}))$ $=\mathsf{gcd}(\mathcal{B}^{(j)}(\mathbf{1}))$ ($1=1,2,\ldots,d$, $j=1,2,\ldots,d$), and the value $z/\mathsf{gcd}(\mathcal{B}^{(i)}(\mathbf{1}))$ is integer. As well $q^{(i)}_{k_i\mathsf{gcd}(\mathcal{B}^{(i)}(\mathbf{1}))}$ is the same for all $i=1,2,\ldots,d$. Hence, with the new notation $\mathsf{gcd}(\cdot)=\mathsf{gcd}(\mathcal{B}^{(i)}(\mathbf{1}))$ relation \eqref{eq.4.4} can be rewritten in the following simplified form
\begin{equation}\label{eq.4.4.1}
D(\mathbf{1})=\sum_{\substack {k_1+k_2+\ldots+k_d=z/\mathsf{gcd}(\cdot)\\k_i\geq0, \quad i=1,2,\ldots,d}}\prod_{i=1}^{d}q_{k_i\mathsf{gcd(\cdot)}}(\mathbf{1}),
\end{equation}
where $q_{k_i\mathsf{gcd(\cdot)}}(\mathbf{1})=q^{(i)}_{k_i\mathsf{gcd(\cdot)}}(\mathbf{1})$ is the new notation.

For the numerator $N(\mathbf{1})$ we have the expression
\begin{equation}
\begin{aligned}\label{eq.4.5}
N(\mathbf{1})&=\sum_{\substack {k_1+k_2+\ldots+k_d=z/\mathsf{gcd}(\cdot)\\k_i\geq0, \quad i=1,2,\ldots,d}}\left(\prod_{i=1}^{d}q_{k_i\mathsf{gcd}(\cdot)}(\mathbf{1})\right.\\
\times&\mathsf{P}\left\{\sum_{i=1}^{d}\bigg(B^{(i)}(\mathbf{1})I_i^+\right. -\widetilde B^{(i)}(\mathbf{1})I_i^-I_i(k_i)\\
&\quad \quad  +\big[\widetilde B^{(i)}(\mathbf{1})-2k_i\mathsf{gcd}(\cdot)\big]I_i^-[1-I_i(k_i)]\bigg)>0\biggl\}\biggl),
\end{aligned}
\end{equation}
where
\begin{eqnarray*}
  I_i^+ &=& \mathsf{I}\{\text{class} \ i \ \text{arrival is positive}\}, \\
  I_i^- &=& \mathsf{I}\{\text{class} \ i \ \text{arrival is negative}\} \quad (I_i^++I_i^-=1),  \\
  I_i(k_i) &=& \mathsf{I}\left\{\widetilde{B}^{(i)}(\mathbf{1})\leq k_i\mathsf{gcd}(\cdot)\right\}.
\end{eqnarray*}
The coefficient 2, that appears in the expression
$$
\big[\widetilde B^{(i)}(\mathbf{1})-2k_i\mathsf{gcd}(\cdot)\big]I_i^-[1-I_i(k_i)]
$$
in the last line of \eqref{eq.4.5} is explained as follows. If immediately before a negative class $i$ arrival  the level of workload was $k_i\mathsf{gcd}(\cdot)$, then after the arrival this level will increase only if the total size of the arrival is greater than $2k_i\mathsf{gcd}(\cdot)$. The value of $k_i\mathsf{gcd}(\cdot)$ of that level will be discharged, and the remaining part that is greater than $k_i\mathsf{gcd}(\cdot)$ will present the new class $i$ workload level.

\subsection{Presentation of \eqref{eq.4.2} (regular case)}\label{S4.3} The limiting probability in \eqref{eq.4.2} is presented in the form
\begin{equation}\label{eq.4.6}
\lim_{t\to\infty}\mathsf{P}\{\|\mathbf{S}_t(a)\|>z~|~\|\mathbf{S}_{t-1}(a)\|=z\}=\frac{N(a)}{D(a)}.
\end{equation}
For the denominator $D(a)$ we have the expression
\begin{equation}\label{eq.4.7}
D(a)=\sum_{\substack {k_1\mathsf{gcd}(\mathcal{B}^{(1)}(a))+k_2\mathsf{gcd}(\mathcal{B}^{(2)}(a))+\ldots+k_d\mathsf{gcd}(\mathcal{B}^{(d)}(a))=z\\k_i\geq0, \quad i=1,2,\ldots,d}}\prod_{i=1}^{d}q^{(i)}_{k_i\mathsf{gcd}(\mathcal{B}^{(i)}(a))}(a).
\end{equation}
According to the presentation $\mathsf{gcd}(\mathcal{B}^{(i)}(a))=a^{(i)}\mathsf{gcd}(\cdot)$, $i=1,2,\ldots$, we have
$$q_{l\mathsf{gcd}(\mathcal{B}^{(i)}(a))}^{(i)}(a)=q_{a^{(i)}l\mathsf{gcd}(\cdot)}^{(i)}(a), \quad l=0,1,\ldots,$$
and according to \eqref{eq.3.8}, for each $i=1,2,\ldots$, we have the identity
$$q_{l\mathsf{gcd}(\mathcal{B}^{(i)}(a))}^{(i)}(a)=q_{l\mathsf{gcd}(\cdot)}(\mathbf{1}), \quad l=0,1,\ldots.$$
Hence,
$$
q_{a^{(i)}l\mathsf{gcd}(\cdot)}^{(i)}(a)=q_{l\mathsf{gcd}(\cdot)}(\mathbf{1}), \quad l=0,1,\ldots.
$$
Then, \eqref{eq.4.7} can be rewritten in the form
\begin{equation}\label{eq.4.8}
\begin{aligned}
D(a)&=\sum_{\substack {a^{(1)}k_1+a^{(2)}k_2+\ldots+a^{(d)}k_d=z/\mathsf{gcd}(\cdot)\\k_i\geq0, \quad i=1,2,\ldots,d}}\prod_{i=1}^{d}q^{(i)}_{a^{(i)}k_i\mathsf{gcd}(\cdot)}(a)\\
&=\sum_{\substack {a^{(1)}k_1+a^{(2)}k_2+\ldots+a^{(d)}k_d=z/\mathsf{gcd}(\cdot)\\k_i\geq0, \quad i=1,2,\ldots,d}}\prod_{i=1}^{d}q_{k_i\mathsf{gcd}(\cdot)}(\mathbf{1}).
\end{aligned}
\end{equation}
For the numerator $N(a)$ we have the expression
\begin{equation}
\begin{aligned}\label{eq.4.9}
N(a)&=\sum_{\substack {a^{(1)}k_1+a^{(2)}k_2+\ldots+a^{(d)}k_d=z/\mathsf{gcd}(\cdot)\\k_i\geq0, \quad i=1,2,\ldots,d}}\left(\prod_{i=1}^{d}q_{k_i\mathsf{gcd}(\cdot)}(\mathbf{1})\right.\\
\times&\mathsf{P}\left\{\sum_{i=1}^{d}\bigg(B^{(i)}(a)I_i^+\right.\\
& \quad\quad-\widetilde B^{(i)}(a)I_i^-\mathsf{I}\{\widetilde B^{(i)}(a)\leq k_i\mathsf{gcd}(\mathcal{B}_i(a))\}\\
& \quad\quad +\big[\widetilde B^{(i)}(a)-2k_i\mathsf{gcd}(\mathcal{B}^{(i)}(a))\big]\\
& \quad\quad\times I_i^-\mathsf{I}\{\widetilde B^{(i)}(a)> k_i\mathsf{gcd}(\mathcal{B}_i(a))\}\bigg)>0\biggl\}\biggl).
\end{aligned}
\end{equation}
Then, assuming that the probability space for the random variables $B^{(i)}(a)$, $\widetilde B^{(i)}(a)$, $B^{(i)}(\mathbf{1})$ and $\widetilde B^{(i)}(\mathbf{1})$, $i=1,2,\ldots, d$, is common, we have the relations
$$
B^{(i)}(a)=a^{(i)}B^{(i)}(\mathbf{1}),
$$
$$
\widetilde B^{(i)}(a)=a^{(i)}\widetilde B^{(i)}(\mathbf{1}).
$$
Substituting these relationships into \eqref{eq.4.9}, we obtain
\begin{equation}
\begin{aligned}\label{eq.4.10}
N(a)&=\sum_{\substack {a^{(1)}k_1+a^{(2)}k_2+\ldots+a^{(d)}k_d=z/\mathsf{gcd}(\cdot)\\k_i\geq0, \quad i=1,2,\ldots,d}}\left(\prod_{i=1}^{d}q_{k_i\mathsf{gcd}(\cdot)}(\mathbf{1})\right.\\
\times&\mathsf{P}\left\{\sum_{i=1}^{d}\bigg(a^{(i)}B^{(i)}(\mathbf{1})I_i^+\right.-a^{(i)}\widetilde B^{(i)}(\mathbf{1})I_i^-I_i(k_i)\\
& \quad\quad \big[a^{(i)}\widetilde B^{(i)}(\mathbf{1})-2a^{(i)}k_i\mathsf{gcd}(\cdot)\big] I_i^-[1-I_i(k_i)]\bigg)>0\biggl\}\biggl).
\end{aligned}
\end{equation}

\subsection{Proof of Theorem \ref{T1} in the case when $\mathbf{x}_t$ is an integer-valued random vector}\label{S4.4}
Consider two series of random variables
\begin{equation}
\begin{aligned}\label{eq.5.1}
&\sum_{i=1}^{d}\bigg(B^{(i)}(\mathbf{1})I_i^+ -\widetilde B^{(i)}(\mathbf{1})I_i^-I_i(k_i)\\
&\quad \quad  +\big[\widetilde B^{(i)}(\mathbf{1})-2k_i\mathsf{gcd}(\cdot)\big]I_i^-[1-I_i(k_i)]\bigg)
\end{aligned}
\end{equation}
and
\begin{equation}\label{eq.5.2}
\begin{aligned}
&\sum_{i=1}^{d}\bigg(a^{(i)}B^{(i)}(\mathbf{1})I_i^+ -a^{(i)}\widetilde B^{(i)}(\mathbf{1})I_i^-I_i(k_i)\\
& \quad\quad +\big[a^{(i)}\widetilde B^{(i)}(\mathbf{1})-2a^{(i)}k_i\mathsf{gcd}(\cdot)\big] I_i^-[1-I_i(k_i)]\bigg).
\end{aligned}
\end{equation}
The first series is associated with the element $\mathbf{1}\in\mathcal{A}$ and the second series is associated with the element $a\in\mathcal{A}$ (in fact an arbitrary element from the set $\mathcal{A}$). The word \textit{series} is related to the positive integer vector of parameters ($k_1$, $k_2$,\ldots,$k_d$).

For any nonnegative integer $k_i$, $i=1,2,\ldots,d$, both of the random variables defined by \eqref{eq.5.1} and \eqref{eq.5.2} have equal expectations, but the variance of the random variable defined by \eqref{eq.5.2} is not smaller than that defined by \eqref{eq.5.1}. Below, we recall and discuss the properties of such random variables.

Recall that for two probability distribution functions $F(x)=\mathsf{P}\{X\leq x\}$ and $G(x)=\mathsf{P}\{Y\leq x\}$ the notation $F\leq_{\mathrm{conv}}G$ means the justice of the inequality
$$
\mathsf{E}[c(X)]\leq\mathsf{E}[c(Y)]
$$
for any convex function $c(x)$ and means that $F(x)$ is smaller than $G(x)$ in convex sense. (The left- and right-hand sides of the inequality are assumed to exist.)

Recall that if $F(x)$ and $G(x)$ are probability distribution functions of positive random variables $X$ and $Y$, respectively, with the properties $\mathsf{E}X=\mathsf{E}Y$ and $\mathsf{var}(X)\leq \mathsf{var}(Y)$, then $F\leq_{\mathrm{conv}}G$.

In the case when $$X=X_1+X_2+\ldots+X_d,$$ $$Y=a^{(1)}X_1+a^{(2)}X_2+\ldots+a^{(d)}X_d,$$ where $X_1$, $X_2$,\ldots, $X_d$ are positive independent identically distributed random variables with finite variance, $a^{(1)}+a^{(2)}+\ldots+a^{(d)}=d$, we have $\mathsf{E}X=\mathsf{E}Y$, $\mathsf{var}(X)\leq \mathsf{var}(Y)$,
and the distribution of $X$ is smaller than the distribution of $Y$ in convex sense.
In this case, we have the following statement.

\begin{lem}\label{L1}
There exists a point $x^*$ such that $\mathsf{P}\{X> x\}\leq\mathsf{P}\{Y> x\}$ for all $x\geq x^*$.
\end{lem}
\begin{proof}
Let $\pi(s)$ denote the Laplace-Stiltjes transform of $X_1$. Then, the Laplace-Stieltjes transform of $X$ is $\pi^d(s)$ and the Laplace-Stieltjes transform of $Y$ is $\prod_{i=1}^{d}\pi\big(a^{(i)}s\big)$. For small positive $s$ we have the expansions
\begin{equation}\label{eq.4.14}
\pi^d(s)=1-sd\mathsf{E}X_1+s^2\left(d\mathsf{var}(X_1)+d^2(\mathsf{E}X_1)^2\right)+o(s^2),
\end{equation}
\begin{equation}\label{eq.4.15}
\begin{aligned}
\prod_{i=1}^{d}\pi\big(a^{(i)}s\big)&=1-sd\mathsf{E}X_1\\
+&s^2\left(\mathsf{var}(X_1)\sum_{i=1}^{d}\left[a^{(i)}\right]^2+d^2(\mathsf{E}X_1)^2\right)
+o(s^2).
\end{aligned}
\end{equation}
Since $\sum_{i=1}^d\big[a^{(i)}]^2\geq d$, then comparing the terms in \eqref{eq.4.14} and \eqref{eq.4.15} we arrive at the conclusion that there exists $s^*>0$ such that for all $0<s\leq s^*$, we have $\pi^d(s)\leq\prod_{i=1}^{d}\pi\big(a^{(i)}s\big)$, where the equality sign holds in the only case of $a^{(i)}\equiv1$ for all $i=1,2,\ldots,d$. Then, inverting the transforms and using the continuity arguments we find that
there exists a point $x^*$ such that for all $x\geq x^*$ the inequality $\mathsf{P}\{X> x\}\leq\mathsf{P}\{Y> x\}$ is true.
\end{proof}

\begin{rem}
For the statements similar to that of Lemma \ref{L1} and the results on convex order relation see
\cite{K}, \cite{M}. Note, that the statement of the theorem remains true if the random variables $X_i$, $i=1,2,\ldots,d$, are not necessarily positive as in the case given by\eqref{eq.5.1} and \eqref{eq.5.2}. In that case, the relation between the probability distributions $\mathsf{P}\{X\leq x\}$ and $\mathsf{P}\{Y\leq x\}$ is not convex in general,
but the proof of Lemma \ref{L1} remains the same with the only difference that instead of Laplace-Stieltjes transforms we would need to use characteristic functions.
\end{rem}

\medskip
Consider now the two conditional expectations
\begin{equation}
\begin{aligned}\label{eq.5.3}
&\mathsf{E}\left\{\sum_{i=1}^{d}\bigg(B^{(i)}(\mathbf{1})I_i^+\right. -\widetilde B^{(i)}(\mathbf{1})I_i^-I_i(k_i^{*})\\
&+\big[\widetilde B^{(i)}(\mathbf{1})-2k_i^{*}\mathsf{gcd}(\cdot)\big]
I_i^-[1-I_i(k_i^{*})]\bigg)~\biggl|~\mathsf{gcd}(\cdot)\sum_{i=1}^{d}k_i^*=Kd\biggl\}
\end{aligned}
\end{equation}
and
\begin{equation}\label{eq.5.4}
\begin{aligned}
&\mathsf{E}\left\{\sum_{i=1}^{d}\bigg(a^{(i)}B^{(i)}(\mathbf{1})I_i^+\right.-a^{(i)}\widetilde B^{(i)}(\mathbf{1})I_i^-I_i(k_i^{*})\\
& \quad\quad +a^{(i)}\big[\widetilde B^{(i)}(\mathbf{1})-2k_i^{*}\mathsf{gcd}(\cdot)\big]\\
& \quad\quad\quad \times\left. I_i^-[1-I_i(k_i^{*})]\right)~\biggl|~\mathsf{gcd}(\cdot)\sum_{i=1}^{d}k_i^{*}=Kd\biggl\},
\end{aligned}
\end{equation}
which are in fact the conditional expectations of random variables similar to those given in \eqref{eq.5.1} and \eqref{eq.5.2}, respectively, with the only difference that instead of a fixed deterministic vector ($k_1$, $k_2$,\ldots,$k_d$) we consider the random vector ($k_1^*$, $k_2^*$,\ldots,$k_d^*$) and the conditioning on $\sum_{i=1}^dk_i^*$. The random vector ($k_1^*$, $k_2^*$,\ldots,$k_d^*$) is defined as follows.
The integer random variables $k_1^*$, $k_2^*$,\ldots, $k_d^*$ are assumed to be independent identically distributed and bounded by an integer value $Kd/\mathsf{gcd}(\cdot)$, i.e. $0\leq k_i^*\leq Kd/\mathsf{gcd}(\cdot)$, where $K$ is now a new series parameter.  Since $a^{(1)}+a^{(2)}+\ldots+a^{(d)}=d$, then it is readily seen that the conditional expectations given by \eqref{eq.5.3} and \eqref{eq.5.4}
are the same, since the expectations of the random variables in \eqref{eq.5.1} and \eqref{eq.5.2} coincide for all fixed vectors ($k_1$, $k_2$,\ldots,$k_d$). However, the conditional variance
\begin{equation*}
\begin{aligned}\label{eq.5.7}
&\mathsf{var}\left\{\sum_{i=1}^{d}\bigg(B^{(i)}(\mathbf{1})I_i^+\right. -\widetilde B^{(i)}(\mathbf{1})I_i^-I_i(k_i^{*})\\
&+\big[\widetilde B^{(i)}(\mathbf{1})-2k_i^{*}\mathsf{gcd}(\cdot)\big]
I_i^-[1-I_i(k_i^{*})]\bigg)~\biggl|~\mathsf{gcd}(\cdot)\sum_{i=1}^{d}k_i^*=Kd\biggl\}
\end{aligned}
\end{equation*}
is not greater than the corresponding conditional variance
\begin{equation*}\label{eq.5.8}
\begin{aligned}
&\mathsf{var}\left\{\sum_{i=1}^{d}\bigg(a^{(i)}B^{(i)}(\mathbf{1})I_i^+\right.-a^{(i)}\widetilde B^{(i)}(\mathbf{1})I_i^-I_i(k_i^{*})\\
& \quad\quad +a^{(i)}\big[\widetilde B^{(i)}(\mathbf{1})-2k_i^{*}\mathsf{gcd}(\cdot)\big]\\
& \quad\quad\quad \times\left. I_i^-[1-I_i(k_i^{*})]\right)~\biggl|~\mathsf{gcd}(\cdot)\sum_{i=1}^{d}k_i^{*}=Kd\biggl\}.
\end{aligned}
\end{equation*}

Our next step is to prove that for large $K$
\begin{equation}
\begin{aligned}\label{eq.5.9}
&\mathsf{P}\left\{\sum_{i=1}^{d}\bigg(B^{(i)}(\mathbf{1})I_i^+\right.
-\widetilde B^{(i)}(\mathbf{1})\}I_i^-I_i(k_i^*)\\
& \quad\quad +\big[\widetilde B^{(i)}(\mathbf{1})-2k_i^*\mathsf{gcd}(\cdot)\big]\\
& \quad\quad \times I_i^-[1-I_i(k_i^*)]\bigg)>0~\biggl|~\mathsf{gcd}(\cdot)\sum_{i=1}^{d}k_i^{*}=Kd\biggl\}
\end{aligned}
\end{equation}
is not greater than
\begin{equation}\label{eq.5.10}
\begin{aligned}
&\mathsf{P}\left\{\sum_{i=1}^{d}\bigg(a^{(i)}B^{(i)}(\mathbf{1})I_i^+\right.-a^{(i)}\widetilde B^{(i)}(\mathbf{1})I_i^-I_i(k_i^*)\\
& \quad\quad +a^{(i)}\big[\widetilde B^{(i)}(\mathbf{1})-2k_i^*\mathsf{gcd}(\cdot)\big]\\
& \quad\quad\times I_i^-[1-I_i(k_i^*)]\bigg)>0~\biggl|~\mathsf{gcd}(\cdot)\sum_{i=1}^{d}k_i^{*}=Kd\biggl\}.
\end{aligned}
\end{equation}

Assume that $K$ increases to infinity, and 
let $K_1$, $K_2$,\ldots,$K_d$ be some values satisfying the equality $K_1+K_2+\ldots+K_d=Kd/\mathsf{gcd}(\cdot)$.
Then, applying the arguments of Lemma \ref{L1} when $K$ is large, one can arrive at the conclusion that the probability
\begin{equation}
\begin{aligned}\label{eq.5.5}
&\mathsf{P}\left\{\sum_{i=1}^{d}\bigg(B^{(i)}(\mathbf{1})I_i^+\right.
-\widetilde B^{(i)}(\mathbf{1})I_i^-I_i(K_i)\\
&\quad \quad  +\big[\widetilde B^{(i)}(\mathbf{1})-2K_i\mathsf{gcd}(\cdot)\big] I_i^-[1-I_i(K_i)]\bigg)>0\biggl\}
\end{aligned}
\end{equation}
is not greater than
\begin{equation}\label{eq.5.6}
\begin{aligned}
&\mathsf{P}\left\{\sum_{i=1}^{d}a^{(i)}\bigg(B^{(i)}(\mathbf{1})I_i^+\right.
-\widetilde B^{(i)}(\mathbf{1})I_i^-I_i(K_i)\\
&\quad \quad  +\big[\widetilde B^{(i)}(\mathbf{1})-2K_i\mathsf{gcd}(\cdot)\big] I_i^-[1-I_i(K_i)]\bigg)>0\biggl\}.
\end{aligned}
\end{equation}

Indeed, \eqref{eq.5.5} and \eqref{eq.5.6} can be, respectively, rewritten in the forms
\begin{equation}
\begin{aligned}\label{eq.5.11}
&\mathsf{P}\left\{\sum_{i=1}^{d}\bigg(\big[B^{(i)}(\mathbf{1})+K_i\mathsf{gcd}(\cdot)\big]I_i^+\right.\\
&\quad\quad -\big[\widetilde B^{(i)}(\mathbf{1})-K_i\mathsf{gcd}(\cdot)\big]I_i^-I_i(K_i)\\
&\quad \quad  +\big[\widetilde B^{(i)}(\mathbf{1})-K_i\mathsf{gcd}(\cdot)\big] I_i^-[1-I_i(K_i)]\bigg)>Kd\biggl\}
\end{aligned}
\end{equation}
and
\begin{equation}
\begin{aligned}\label{eq.5.12}
&\mathsf{P}\left\{\sum_{i=1}^{d}a^{(i)}\bigg(\big[B^{(i)}(\mathbf{1})+K_i\mathsf{gcd}(\cdot)\big]I_i^+\right.\\
&\quad\quad -\big[\widetilde B^{(i)}(\mathbf{1})-K_i\mathsf{gcd}(\cdot)\big]I_i^-I_i(K_i)\\
&\quad \quad  +\big[\widetilde B^{(i)}(\mathbf{1})-K_i\mathsf{gcd}(\cdot)\big] I_i^-[1-I_i(K_i)]\bigg)>Kd\biggl\}
\end{aligned}
\end{equation}
Note, that the expressions in \eqref{eq.5.11} and \eqref{eq.5.12} characterize the probabilities that the corresponding positive random variables are greater than $Kd$. This means that the probability distribution associated with \eqref{eq.5.11} is less than that probability distribution associated with \eqref{eq.5.12} in convex sense. To apply Lemma \ref{L1} for large $K$ note as follows. Let $\mathcal{I}$ denote the set of indices $i$ for which $I^+_i=1$, let $\mathcal{J}$ denote the set of indices $i$ for which $I^-I_i(K_i)=1$ and let $\mathcal{K}$ denote the set of indices $i$ for which $I^-[1-I_i(K_i)]=1$. The sets $\mathcal{I}$, $\mathcal{J}$ and $\mathcal{K}$ are disjoint, and $\mathcal{I}\cup\mathcal{J}\cup\mathcal{K}$ $=\{1,2,\ldots,d\}$.
Then, \eqref{eq.5.11} and \eqref{eq.5.12}, respectively, can be rewritten by
\begin{equation}\label{eq.5.13}
\begin{aligned}
&\mathsf{P}\left\{\sum_{i\in\mathcal{I}}B_i(\mathbf{1})-\sum_{i\in\mathcal{J}}\widetilde B_i(\mathbf{1})+\sum_{i\in\mathcal{K}}\widetilde B_i(\mathbf{1})\right.\\
&\left.+\mathsf{gcd}(\cdot)\left[\sum_{i\in\mathcal{I}\cup\mathcal{J}}K_i-\sum_{i\in\mathcal{K}}K_i\right]>Kd\right\}
\end{aligned}
\end{equation}
and
\begin{equation}\label{eq.5.14}
\begin{aligned}
&\mathsf{P}\left\{\sum_{i\in\mathcal{I}}a^{(i)}B_i(\mathbf{1})-\sum_{i\in\mathcal{J}}a^{(i)}\widetilde B_i(\mathbf{1})+\sum_{i\in\mathcal{K}}a^{(i)}\widetilde B_i(\mathbf{1})\right.\\
&\left.+\mathsf{gcd}(\cdot)\left[\sum_{i\in\mathcal{I}\cup\mathcal{J}}a^{(i)}K_i-\sum_{i\in\mathcal{K}}a^{(i)}K_i\right]>Kd\right\}.
\end{aligned}
\end{equation}

Now Lemma \ref{L1} can be applied. If $\sum_{i\in\mathcal{K}}K_i$ is an empty sum
with probability 1,
that is, for all $K_i$, $i=1,2,\ldots,d$, we have $\mathsf{P}\{\widetilde{B}_i\leq 2K_i\}=1$,
then the probabilities given by \eqref{eq.5.13} and \eqref{eq.5.14} are identical and both equal to half, since $I_i^+=1$ and $I_i^-=1$ have probability half.
So, the required result, conditioned by $(K_1, K_2,\ldots, K_d)$ given, trivially follows.  If $\sum_{i\in\mathcal{K}}K_i$ is not an empty sum, then
$$
Kd-\mathsf{gcd}(\cdot)\left[\sum_{i\in\mathcal{I}\cup\mathcal{J}}K_i-\sum_{i\in\mathcal{K}}K_i\right]
$$
and
$$
Kd-\mathsf{gcd}(\cdot)\left[\sum_{i\in\mathcal{I}\cup\mathcal{J}}a^{(i)}K_i-\sum_{i\in\mathcal{K}}a^{(i)}K_i\right]
$$
are large, and by applying Lemma \ref{L1} we arrive at the required inequality, finally, due to the total probability formula, which implies different and interchangeable values $K_i$, since $q_{l\mathsf{gcd}(\cdot)}^{(i)}$ are the same for all $i=1,2,\ldots,d$.

Let us now return to relations \eqref{eq.4.3}, \eqref{eq.4.4.1} and \eqref{eq.4.5} in simple case (Section \ref{S4.2}) and to relations \eqref{eq.4.6}, \eqref{eq.4.8} and \eqref{eq.4.10} in regular case (Section \ref{S4.3}). For large $K$, take $(k_1,k_2,\ldots,k_d)$=$(K_1,K_2,\ldots,K_d)$,
where $K_1$, $K_2$,\ldots,$K_d$ are large numbers satisfying $K_1+K_2+\ldots+K_d=Kd/\mathsf{gcd}(\cdot)$. Then, the denominators given by \eqref{eq.4.4.1} and by the right-hand side of \eqref{eq.4.8} are asymptotically identical, since the number of all possible values of $(K_1$, $K_2$,\ldots, $K_d)$ and $(a^{(1)}K_1$, $a^{(2)}K_2$,\ldots, $a^{(d)}K_d)$ asymptotically coincide (recall that $a^{(1)}+a^{(2)}+\ldots+a^{(d)}=d$).

To compare now the fractions given by the right-hand sides of \eqref{eq.4.3} and \eqref{eq.4.6} we use the following elementary fact. Let $R_1$ and $R_2$ be two fractions of the form
$$
R_1=\frac{b_1^{(1)}+b_2^{(1)}+\ldots+b_n^{(1)}}{c_1^{(1)}+c_2^{(1)}+\ldots+c_n^{(1)}}
$$
and
$$
R_2=\frac{b_1^{(2)}+b_2^{(2)}+\ldots+b_n^{(2)}}{c_1^{(2)}+c_2^{(2)}+\ldots+c_n^{(2)}}
$$
with positive $b_j^{(1)}$, $b_j^{(2)}$, $c_j^{(1)}$ and $c_j^{(2)}$, $j=1,2,\ldots,n$.
If the inequality $b_j^{(1)}/c_j^{(1)}\leq b_j^{(2)}/c_j^{(2)}$ is true for all $j=1,2,\ldots,n$, then $R_1\leq R_2$.

Applying this, we arrive at the conclusion that the right-hand side of \eqref{eq.4.3} is not greater than the right-hand side of \eqref{eq.4.6}, and, hence, the inequality between \eqref{eq.5.9} and \eqref{eq.5.10} is equivalent to the inequality
\begin{equation*}
\begin{aligned}
&\lim_{t\to\infty}\mathsf{P}\{\|\mathbf{S}_t(\mathbf{1})\|>z~|~\|\mathbf{S}_{t-1}(\mathbf{1})\|=z\}\\
&\leq\lim_{t\to\infty}\mathsf{P}\{\|\mathbf{S}_t(a)\|>z~|~\|\mathbf{S}_{t-1}(a)\|=z\},\\
& \quad a\in\mathcal{A}
\end{aligned}
\end{equation*}
for all $z\geq z^*$, where $z^*$ is some (possibly large) value.
This proves the theorem in the case when $\mathbf{x}_t$ is an integer-valued vector.

\subsection{Extension of the proof of Theorem \ref{T1} in the case when $\mathbf{x}_t$ is a real-valued random vector}\label{S4.5}

We consider the series of random variables $B^{(i)}(a,m)$ and $\widetilde B^{(i)}(a,m)$, where an added parameter $m$ is the series parameter, $i=1,2,\ldots,d$ and $a\in\mathcal{A}$. As before, the random variables $B^{(i)}(\mathbf{1},m)$ and $\widetilde B^{(i)}(\mathbf{1},m)$ are of special significance. Assume that these random variables take values from the set of rational numbers with span $\alpha_m$, denoted as $\mathsf{span}_m(\cdot)$. Then, similarly to the consideration above, the random variables $B^{(i)}(a,m)$ and $\widetilde B^{(i)}(a,m)$ have span $a^{(i)}\mathsf{span}_m(\cdot)$. Apparently, that all the arguments provided for integer-valued random variables are applicable here as well, and the required proofs are true with the replacement of $\mathsf{gcd}(\cdot)$ by $\mathsf{span}_m(\cdot)$ for each of the $m$th series indexed with $m$. Then the statement of the theorem in the continuous case follows by taking the appropriate limit as $\mathsf{span}_m(\cdot)\to0$ with the series parameter $m$ increasing to infinity.

\section{Concluding remarks}\label{S5}
In the present paper, we studied semiconservative random walks in week sense. Originally we attempted to construct semiconservative random walks given by Definition \ref{D1}, but we found the problem hard. So, the problem to describe new nontrivial classes of conservative and semiconservative random walks in $\mathbb{Z}^d$ or $\mathbb{R}^d$ is open.

\end{document}